\documentclass[11pt, reqno]{amsart}

\usepackage{amsmath,amssymb,amsthm, epsfig}
\usepackage{hyperref}
\usepackage{amsmath}
\usepackage{amssymb}
\usepackage{color}
\usepackage{ulem}
\usepackage{epsfig}
\usepackage[mathscr]{eucal}
\usepackage[latin1]{inputenc}

\newtheorem{theorem}{Theorem}
\newtheorem{definition}{Definition}

\newtheorem{lemma}{Lemma}

\newtheorem{corollary}{Corollary}
\newtheorem{remark}{Remark}

\date{}
\numberwithin{equation}{section}
\numberwithin{theorem}{section}
\numberwithin{lemma}{section}
\numberwithin{corollary}{section}
\numberwithin{remark}{section} 
\numberwithin{proposition}{section}
\numberwithin{definition}{section}

\def \N {\mathbb{N}}
\def \R {\mathbb{R}}

\def \dist {\mathrm{dist}}

\begin{document}

\title[Reaction-diffusion equations for the infinity Laplacian]{Reaction-diffusion equations\\ for the infinity Laplacian}

\author[N.M.L. Diehl]{Nicolau M.L. Diehl}
\address{Instituto Federal de Educa\c c\~ao,  Ci\^encia e Tecnologia do Rio Grande do Sul, Canoas, Brazil}
\email{nicolau.diehl@canoas.ifrs.edu.br}

\author[R. Teymurazyan]{Rafayel Teymurazyan}
\address{CMUC, Department of Mathematics, University of Coimbra, 3001-501 Coimbra, Portugal}
\email{rafayel@utexas.edu}

\begin{abstract}
We derive sharp regularity for viscosity solutions of an inhomogeneous infinity Laplace equation across the free boundary, when the right hand side does not change sign and satisfies a certain growth condition. We prove geometric regularity estimates for solutions and conclude that once the source term is comparable to a homogeneous function, then the free boundary is a porous set and hence, has zero Lebesgue measure. Additionally, we derive a Liouville type theorem. When near the origin the right hand side grows not faster than third degree homogeneous function, we show that if a non-negative viscosity solution vanishes at a point, then it has to vanish everywhere.
\bigskip

\noindent \textbf{Keywords:} Infinity Laplacian, regularity, dead-core problems, porosity.\\
\noindent \textbf{AMS Subject Classification (2010):} 35B09, 35B53, 35B65, 35R35.

\end{abstract}

\maketitle

\section{Introduction}\label{s1}
Reaction-diffusion equations arise naturally when modeling certain phenomena in biological, chemical and physical systems. In this paper we study reaction-diffusion equations for infinity Laplacian, which despite being too degenerate to realistically represent a physical diffusion process, has been studied in the framework of optimization and free boundary problems (see, for example, \cite{ATU19}, \cite{RST17}, \cite{RT12,RTU15,TU17}, just to cite a few). More precisely, we establish regularity and geometric properties of solutions of the problem
\begin{equation}\label{1.1}
\Delta_\infty u=f(u) \quad \text{in} \quad \Omega,
\end{equation}
where $\Omega\subset\R^n$, $f\in C(\R_+)$ and
\begin{equation}\label{1.2}
0\le f(\delta t)\leq M\delta^\gamma f(t),
\end{equation}
with $M>0$, $\gamma\in[0,3)$, $t>0$ bounded, and $\delta>0$ small enough. Additionally, we assume that 
\begin{equation}\label{comparison}
f \text{ is non-decreasing }.
\end{equation}
Here, $\R_+$ is the set of non-negative numbers, and the infinity Laplacian is defined as follows:
$$
\Delta_\infty u(x):=\sum_{i,j=1}^nu_{x_i}u_{x_j}u_{x_ix_j},
$$
with $u_{x_i}=\partial u/\partial x_i$. Note that the continuity of $f$ provides that $f(u)$ is bounded once $u$ is bounded. Note also that \eqref{1.2} is quite general in the sense that it needs to hold only for $\delta$ close to zero. For example, it holds for functions that are homogeneous of degree $\gamma$. Condition \eqref{comparison} is needed to guarantee the comparison principle. Solutions of \eqref{1.1} are understood in the viscosity sense according to the following definition:
\begin{definition}\label{d1.1}
	A function $u\in C(\Omega)$ is called a viscosity super-solution (resp. sub-solution) of \eqref{1.1}, and written as $\Delta_\infty u\le f(u)$ (resp. $\ge$), if for every $\phi\in C^2(\Omega)$ such that $u-\phi$ has a local minimum at $x_0\in\Omega$, with $\phi(x_0)=u(x_0)$, we have
	$$
	\Delta_\infty\phi(x_0)\leq f(\phi(x_0)).\quad \textrm{(resp. $\geq$)}
	$$
	A function $u$ is called a viscosity solution if it is both a viscosity super-solution and a viscosity sub-solution.
\end{definition}

The infinity Laplace operator is related to the absolutely minimizing Lipschitz extension problem: for a given Lipschitz function on the boundary of a bounded domain, find its extension inside the domain in a way that has the minimal Lipschitz constant, \cite{A67}. It is known (see \cite{J93}) that such function $u$ has to be an infinity harmonic one, i.e. $\Delta_\infty u=0$ (in the viscosity sense). The regularity issue of infinity harmonic functions received extensive attention over the years. As was shown in \cite{ES08}, the infinity harmonic functions in the plane are $C^{1,\alpha}$, for a small $\alpha$ (it is conjectured that the optimal regularity is $C^{1,\frac{1}{3}}$). In higher dimensions infinity harmonic functions are known to be everywhere differentiable (see \cite{ES11}). 

As for the inhomogeneous case of $\Delta_\infty u=f$, it is known that the Dirichlet problem has a unique viscosity solution, provided $f$ does not change sign (see \cite{LW08}). Moreover, as was shown in \cite{L14}, for bounded right hand side, the Lipschitz estimate and everywhere differentiability of solutions remain true. The case of $f$ not being bounded away from zero, mainly, when $f=u_+^\gamma$, where $u_+:=\max\left(u,0\right)$ and $\gamma\in[0,3)$ is a constant, was studied in \cite{ALT16} (dead-core problem). The authors show that for such right hand side (strong absorbtion) across the free boundary $\partial\{u>0\}$ non-negative viscosity solutions are of class $C^{\frac{4}{3-\gamma}}$. The denominator $3-\gamma$ is related to the degree of homogeneity of the operator, which is three, i.e., $\Delta_\infty(Cu)=C^3\Delta_\infty u$, for any constant $C$. Note that for $\gamma\in(0,3)$ this regularity is more than the conjectured $C^{1,\frac{1}{3}}$, i.e., we obtain higher regularity across the free boundary. This result allows to establish Hausdorff dimension estimate for the free boundary $\partial\{u>0\}$ and conclude that it has Lebesgue measure zero. 

We extend these results for the source term $f$ satisfying \eqref{1.2}. In particular, it includes equations with the right hand side 
$$
f(t)=e^t-1\,\,\,\textrm{ and }\,\,\,f(t)=\log(t^2+1)
$$
among others (see Section \ref{s7} for more examples). In fact, our results are true in a broader context, when allowing ``coefficients'' in the right hand side, that is, when in \eqref{1.1} one has $f=f(x,u)$, as long as $f(x,u)$ satisfies \eqref{1.2} as a function of $u$ and is continuous (and bounded) as a function of $x$ (see Section \ref{s7}). For simplicity, we restrict ourselves to the case of $f(x,u)=f(u)$.

Our strategy is the following: by means of a flattening argument, we show that across the free boundary $\partial\{u>0\}\cap\Omega$ non-negative viscosity solutions of \eqref{1.1} are of class $C^{\frac{4}{3-\gamma}}$, when \eqref{1.2} holds. When the source term is comparable to a homogeneous function of degree $\gamma$, this result is sharp in the sense that across the free boundary non-negative viscosity solutions grow exactly as $r^{\frac{4}{3-\gamma}}$ in the ball of radius $r$. We also analyze the borderline (critical) case, that is, when $\gamma=3$ (which is also the degree of the homogeneity of the infinity Laplacian). Unlike \cite{ALT16}, $f$ is not given explicitly, which makes it harder to construct a barrier function - needed for our analysis. Nevertheless, we are able to show that in this case \eqref{1.1} has a viscosity sub-solution whose gradient has modulus separated from zero. We use this function to build up a suitable barrier to conclude that if a viscosity solution vanishes at a point, it has to vanish everywhere. Our results remain true when the right hand side has some ``bounded coefficients'' (see Remark \ref{r7.1}). For simplicity we restrict ourselves with the right hand side ``without coefficients''.

The paper is organized as follows: in Section \ref{s2}, we prove an auxiliary result (flattening solutions) (Lemma \ref{l2.2}), which we use in Section \ref{s3} to derive the main regularity result (Theorem \ref{t3.1}), and as a consequence, in Section \ref{Liouville}, we obtain Liouville type theorems (Theorem \ref{t3.2} and Theorem \ref{t3.3}). In Section \ref{s4}, we prove several geometric measure estimates (Theorem \ref{t4.1} (non-degeneracy) and Corollary \ref{c4.1} (porosity)), and conclude that the free boundary has Lebesgue measure zero (Corollary \ref{c4.2}). In Section \ref{s5}, when $\gamma=3$, we show that the only non-negative viscosity solution that has zero, is the function that is identically zero (Theorem \ref{t5.1}). Finally, in Section \ref{s7} we bring some examples of source terms for which our results are true.

\section{Preliminaries}\label{s2}
In this section we list some preliminaries, as well as prove an auxiliary lemma for future reference. We start by the comparison principle, the proof of which can be found in \cite{CIL92,LW08}.
\begin{lemma}\label{l2.1}
	Let $u$, $v\in C(\overline{\Omega})$ be such that
	$$
    \Delta_\infty u-f(u)\le0, \,\,\,\Delta_\infty v-f(v)\ge0\,\,\textrm{ in }\,\,\Omega
	$$
	in the viscosity sense, and $f$ satisfy \eqref{comparison} or $\inf f>0$. If $u\geq v$ on $\partial\Omega$, then $u\geq v$ in $\Omega$.
\end{lemma}
The comparison principle, together with Perron's method leads to the following result (for the proof we refer the reader to \cite{CIL92}, for example). In fact, existence of solutions can be shown even without directly applying the comparison principle, as it was done, for example, in \cite[Theorem 3.1]{RTU19}.
\begin{theorem}\label{t2.1}
	If $\Omega\subset\R^n$ is bounded and $\varphi\in C(\partial\Omega)$ is a non-negative function, then there is a unique and non-negative function $u$ that solves the Dirichlet problem
	\begin{equation}\label{2.1}
	\left\{
	\begin{aligned}
		\Delta_\infty u&=f(u) \text { in } \Omega, \\
		u&=\varphi\text { on } \partial\Omega
	\end{aligned}\right.
	\end{equation}
in the viscosity sense.	
\end{theorem}
The following auxiliary lemma is a variant of the flatness improvement technique introduced in \cite{ALT16,T13,T16} to study the regularity properties of solutions of dead-core problems.
\begin{lemma}\label{l2.2} Let $g\in L^\infty(B_1)\cap C(B_1)$ be a non-negative function such that 
	$$
	\|g\|_\infty\le\max\{1,M\}\sup_{[0,\|u\|_\infty]}f,
	$$
	where $f$ and $M$ are as in \eqref{1.2}. For any given $\mu >0 $ there exists a constant $\kappa _{\mu}=\kappa(\mu,n)>0$ such that if in $B_1$ a continuous functions $v$, which vanishes at the origin and $v\in[0,1]$, satisfies, in viscosity sense,
	$$\Delta_\infty v  - \kappa_\mu^4g(v)   = 0$$\\
	for $0<\kappa\leq \kappa_{\mu}$, then \\
	$$\sup_{B_{1/2}} v \leq \mu.$$
\end{lemma}
\begin{proof}
We argue by contradiction assuming that there exist $\mu^* >0$, $\{v_i\}_{i\in \N}$ and $\{\kappa_i\}_{i\in \N}$ with $v_i(0)=0$, $0\leq v_i\leq1$, in $B_1$ satisfying in viscosity sense to
	$$\Delta_\infty v_i-\kappa_i^4 g(v_i)= 0$$\\
	where $\kappa_i = \text{o}(1)$, while
\begin{equation}\label{2.2}					
\sup_{B_{1/2}} v_i > \mu^*.
\end{equation}\\
By local Lipschitz regularity (see \cite[Corollary 2]{L14}, for example), the sequence $\{v_i\}_{i\in \N}$ is pre-compact in the $C^{0,1}(B_{3/4})$. Hence, by Arzel\`{a}-Ascoli theorem, $v_i$ converges (up to a subsequence) to a function $v_\infty$ locally uniformly in $B_{2/3}$. Moreover, $v_\infty(0)=0, \; 0\leq v_\infty \leq 1  $ and $\Delta_\infty v_\infty = 0.$ The maximum principle for the infinity harmonic functions then yields $v\equiv  0$, which contradicts to \eqref{2.2} once $i$ is big enough.
\end{proof}
The following definition is for future reference.
\begin{definition}\label{d1.2}
	A function $u$ is called an entire solution, if it is a viscosity solution of \eqref{1.1} in $\R^n$.
\end{definition}
We close this section by reminding the notion of porosity.
\begin{definition}\label{porosity}
	The set $E\subset\R^n$ is called porous with porosity $\sigma$, if there is $R>0$ such that $\forall x\in E$ and $\forall r\in (0,R)$ there exists $y\in\R^n$ such that
	$$
	B_{\sigma r}(y)\subset	B_{r}(x)\setminus E.
	$$
\end{definition}
	A porous set of porosity $\sigma$ has Hausdorff dimension not exceeding
$n-c\sigma^n$, where $c>0$ is a constant depending only on dimension. In particular, a porous set has Lebesgue measure zero (see \cite{Z88}, for instance).

\section{Regularity across the free boundary}\label{s3}
In this section we make use of Lemma \ref{l2.2} and derive regularity result for viscosity solutions of \eqref{1.1} across the free boundary $\partial\{u>0\}$.
\begin{theorem}\label{t3.1}
	If $u$ is a non-negative viscosity solution of \eqref{1.1}, where $f$ satisfies \eqref{1.2}, and $x_0\in\partial \{u>0\}\cap\Omega$, then there exists a constant $C>0$, depending only on $\gamma$, $\|u\|_\infty$ and $\dist (x_0 , \partial \Omega)$, such that
	$$u(x) \leq C |x - x_0|^{\frac{4}{3-\gamma}}$$
	for $x \in \{u>0\} $ near $x_0$.
\end{theorem}
\begin{proof}
	The idea is to use an iteration argument and carefully choose sequence of functions that allows to make use of the Lemma \ref{l2.2}.  Observe that without loss of generality, we may assume that $x_0=0$ and $B_1\subset\Omega$.
	
	For $\mu =2^{-\frac{4}{3-\gamma}}$, let now $\kappa_{\mu} >0$ be as in Lemma \ref{l2.2}. We then construct the first member of the sequence by setting	
	$$
	w_0(x):= \tau u(\rho x) \quad \text{in} \quad B_1,
	$$	
	where
	$$
	\tau:= \min \left\{1, \|u\|_\infty^{-1} \right\}\,\,\,\textrm{ and }\,\,\,\rho :=\kappa_\mu   \tau^{-\frac{3-\gamma}{4}}.
	$$
	Note that $\tau^3\rho^4=\kappa_\mu^4\tau^\gamma$, $w_0(0)=0$ and in $w_0\in[0,1]$. Since $u$ is a viscosity solution of \eqref{1.1}, then
	$$
	\Delta_\infty w_0(x) - \tau^3\rho^4 f(\tau^{-1}w_0(x))=0
	$$
	or, equivalently,
	\begin{equation}\label{3.1}
	\Delta_\infty w_0(x)-\kappa_\mu^4\tau^{ \gamma} f(\tau^{-1}w_0(x)) = 0.
	\end{equation}
	Since $\tau\le1$, then 	
	$g(w_0):=\tau^\gamma f(\tau^{-1}w_0)\le f(u(\rho x))\le\displaystyle\sup_{[0,\|u\|_\infty]}f$. From Lemma \ref{l2.2}, we obtain	
	$$
	\sup_{B_{1/2}} w_0 \leq 2^{-\frac{4}{3-\gamma}}.
	$$
	For $i\in \N $, we then define
	\[ w_i(x) := 2^{ -\frac{4}{3-\gamma}}w_{i-1}(2^{-1}x). \]
	and observe that $w_i(0)=0$, $w_i\in[0,1]$ and $w_i$ satisfies 
	$$
	\Delta_\infty w_i(x)=\kappa_\mu^4 2^{\frac{4\gamma}{3-\gamma}i}\tau^\gamma f\left(\tau^{-1}2^{-\frac{4}{3-\gamma}i}w_i(x)\right).
	$$
	Using \eqref{1.2}, for $i$ big we estimate
	$$
	2^{\frac{4\gamma}{3-\gamma}i}\tau^\gamma f\left(\tau^{-1}2^{-\frac{4}{3-\gamma}i}w_i(x)\right)\le M\tau^\gamma f\left(\tau^{-1}w_i(x)\right)\le M\displaystyle\sup_{[0,\|u\|_\infty]}f.
	$$
	Once again applying Lemma \ref{l2.2}, one gets
	\[  \sup_{B_{1/2}} w_i \leq 2^{-\frac{4}{3-\gamma}},  \]
	or in other terms,
	\[  \sup_{B_{1/4}} w_{i-1} \leq 2^{-2 \frac{4}{3-\gamma}}. \]
	Continuing this way, for $w_0$ we obtain
	\begin{equation}\label{3.2}
	\sup_{B_{2^{-i}}} w_0 \leq 2^{-i\frac{4}{3-\gamma}}.
	\end{equation}
	Next, for a fixed $0<r\leq \frac{\rho}{2}$, by choosing $i\in \N$ such that
	$$
	2^{-(i+1)} < \frac{r}{\rho}  \leq 2^{-i},
	$$
	and using \eqref{3.2}, we estimate
	\begin{equation*}
	\begin{split}
    \sup_{B_{r}} u &\leq \sup_{B_{\rho2^{-i}}} u = \tau^{-1}\sup_{B_{\rho2^{-i}}} w_0\\
    &\leq\tau^{-1} 2^{-i\frac{4}{3-\gamma}} = 2^{\frac{4}{3-\gamma}}\tau^{-1}2^{-(i+1)\frac{4}{3-\gamma}}\\
    & \leq \left(\tau^{-1}2\rho^{-1}\right)^{\frac{4}{3-\gamma}}r^{\frac{4}{3-\gamma}}\\
    &=Cr^{\frac{4}{3-\gamma}}.
	\end{split}
	\end{equation*}
\end{proof}
Geometrically Theorem \ref{t3.1} means that no matter how ``bad'' the function $u$ is in $\{u>0\}$, it touches the free boundary $\partial\{u>0\}$ smoothly. In other words, a non-negative viscosity solution of \eqref{1.1} may have cusp singularities in its positivity set, and yet it is smooth near its free boundary. 

\section{Liouville type results}\label{Liouville}

Despite the regularity information being available only across the free boundary, it is enough to derive the following Liouville type theorem.
\begin{theorem}\label{t3.2}
	If $u$ is an entire solution, \eqref{1.2} holds and $u(x_0)=0$ for a $x_0\in\R^n$ with
    \begin{equation}\label{3.3}
    	u(x) = o\left(|x|^{\frac{4}{3-\gamma}}\right),\,\,\,\textrm{ as }\,\,\,|x| \rightarrow \infty,
    \end{equation}
    then $u \equiv0$.
\end{theorem}
\begin{proof}
	Without loss of generality we may assume that $x_0=0$. For $k\in\mathbb{N}$, set
	\[ u_k(x):= k^{\frac{-4}{3-\gamma}}u(kx),\quad x\in B_1,\]
	where $B_1$ is the ball of radius one centered at the origin. Note that $ u_k(0)=0 $. Since $u$ is an entire solution, for $x\in B_1$ one has
	\begin{equation*}
	\Delta_\infty u_k(x)=k^{\frac{-4 \gamma}{3-\gamma}} f\left(k^{\frac{4}{3-\gamma}} u_{k}(x)\right).
	\end{equation*}	
	Note that the right hand side of the last equation satisfies \eqref{1.2}. From Theorem \ref{t3.1}, we then deduce that if $x_k\in\overline{B}_r$ is such that
    $$
    u_k(x_k)=\sup_{\overline{B}_r} u_k,
    $$
    where $r>0$ is small, then in $B_r$ one has
    \begin{equation}\label{3.4}
    \|u_k \|_\infty\to0,\,\,\,\textrm{ as }\,\,\,k\to\infty.
    \end{equation}
    In fact, if $|kx_k|$ remains bounded as $k\to\infty$, then applying Theorem \ref{t3.1} to $u_k$ we obtain
	\begin{equation}\label{3.5}
	u_k (x_k) \le C_k|x_k|^{\frac{4}{3-\gamma}},
	\end{equation}
	where $C_k>0$ and $C_k\to0$. This implies that $u(kx_k)$ remains bounded as $k\to\infty$, and therefore $u_k(x_k)\to0$, as $k\to\infty$, and \eqref{3.4} is true. It remains true also in the case when $|kx_k|\to\infty$, as $k\to\infty$, since then from \eqref{3.3} we get
	\begin{equation*}
	u_k(x_k) \leq |kx_k|^{-\frac{4}{3-\gamma}}k^{-\frac{4}{3-\gamma}}\to0,\,\,\,\textrm{ as }\,\,\,k\to\infty.
	\end{equation*}
	Now, if there exists $y\in\R^n$ such that $u(y)>0$, by choosing $k\in\mathbb{N}$ large enough so $y\in B_{kr}$ and using \eqref{3.4} and \eqref{3.5}, we estimate
	\[\dfrac{u(y)}{|y|^{\frac{4}{3-\gamma}}} \leq  \sup_{B_{kr}} \dfrac{u(x)}{|x|^{\frac{4}{3-\gamma}}}=\sup_{B_r}  \dfrac{u_k(x)}{|x|^{\frac{4}{3-\gamma}}}\leq \dfrac{u(y)}{2|y|^{\frac{4}{3-\gamma}}},\]
	which is a contradiction.
\end{proof}
In fact, once the comparison principle holds, the condition \eqref{3.3} can be weakened in the following sense (Theorem \ref{t3.3} below). Let $x_0\in\R^n$ and $r>0$ be fixed, and let $u\ge0$ be the unique solution of \eqref{2.1} in $B_r(x_0)$ with $\varphi\equiv\alpha_r>0$ constant, guaranteed by Theorem \ref{t2.1}. Note that $u$ is a viscosity sub-solution of
\begin{equation}\label{3.7}
\left\{
\begin{aligned}
\Delta_\infty v &= \lambda v_+^{\gamma} &&\textrm{ in } B_r(x_0),\\
v&=\alpha_r &&\textrm{ on } \partial B_r(x_0),
\end{aligned}
\right.
\end{equation}
where
\begin{equation}\label{3.6}
\lambda:=M^{-1}\beta^{-\gamma}f(\beta),
\end{equation}
and $\beta>\|u\|_\infty$ is a constant big enough so \eqref{1.2} holds. Then the condition \eqref{3.3} can be weakened and substituted by
\begin{equation}\label{3.8}
	\limsup_{|x| \rightarrow \infty} \dfrac{u(x)}{|x-x_0|^{\frac{4}{3-\gamma}}} < \left( \lambda \frac{(3-\gamma)^4}{64(1+\gamma)} \right)^{\frac{1}{3-\gamma}},
\end{equation}
where $\lambda$ is defined by \eqref{3.6}, and Theorem \ref{t3.2} can be improved to the following variant (see Theorem \ref{t3.3} below). The choice of the right hand side of \eqref{3.8} comes from the explicit structure of the unique solution of \eqref{3.7}, which, as observed in \cite{ALT16}, is given by
\begin{equation}\label{3.9}
v(x):=\Upsilon\left(|x-x_0|-r + \left(\frac{\alpha_r}{\Upsilon}\right)^{\frac{3-\gamma}{4}} \right)^{\frac{4}{3-\gamma}}_+,
\end{equation}
where
\begin{equation}\label{3.10}
	\Upsilon:= \left(\lambda \frac{(3-\gamma)^4}{64(1+\gamma)}\right)^{\frac{1}{3-\gamma}}.
\end{equation}
\begin{theorem}\label{t3.3}
	Let \eqref{1.2}, \eqref{comparison} hold. If $u$ is an entire solution and satisfies \eqref{3.8}, then $u\equiv0$.
\end{theorem}
\begin{proof}
	Once $r>0$ is large enough, then \eqref{3.8} guarantees, with $\Upsilon>0$ defined by \eqref{3.10},
	\[ \sup_{\partial B_r} \dfrac{u(x)}{r^{\frac{4}{3-\gamma}}} \leq \theta\Upsilon,\]
	for some $\theta < 1$. On the other hand, using \eqref{1.2}, one has that the unique solution of \eqref{3.7}, with $\alpha_r=\displaystyle\sup_{\partial B_r(x_0)}u$, given by \eqref{3.9}, is a viscosity sub-solution of \eqref{1.1}. The comparison principle, Lemma \ref{l2.1}, then implies that $u\le v$ in $B_r(x_0)$. Letting $r\to\infty$, we conclude that $u\equiv0$.	
\end{proof}
\begin{remark}\label{r3.1}
	As can be seen from \eqref{3.9}, the plateau of $v$, i.e., the set $\{v=0\}$, is the ball $\overline{B}_{R}(x_0) $, where
	$$
	0<R:=r - \left(\frac{\alpha_r}{\Upsilon}\right)^{\frac{3-\gamma}{4}}.
	$$
    Since $0\le u\le v$, the plateau of $u$ contains the $\overline{B}_R(x_0)$.
\end{remark}

\begin{remark}\label{r3.2}
	Note that the inequality \eqref{3.8} has to be strict. For example, if
	$$
	w(x):=\Upsilon|x-x_0|^{\frac{4}{3-\gamma}}
	$$
	then
	$$
	\limsup_{|x| \rightarrow \infty} \dfrac{w(x)}{|x-x_0|^{\frac{4}{3-\gamma}}}=\Upsilon,
	$$
	but $w$ is not identically zero.
\end{remark}

\section{Non-degeneracy and porosity}\label{s4}

In this section we show that once
\begin{equation}\label{5.1}
	f(\delta t)\ge N\delta^\gamma f(t)\ge0,
\end{equation}
with $N>0$, $\gamma\in[0,3)$, $t>0$ bounded, and $\delta>0$ small enough, then across the free boundary non-negative viscosity solutions of \eqref{1.1} grow exactly as $r^\frac{4}{3-\gamma}$ in the ball $B_r$, for $r>0$ small enough. As a consequence, we conclude that the touching ground surface is a porous set, which implies that it has Hausdorff dimension less than $n$, and so its Lebesgue measure is zero (see \cite{Z88}). We start by the following non-degeneracy theorem.
\begin{theorem}\label{t4.1} 
	Let \eqref{5.1} hold. Let also $f$ satisfy \eqref{comparison} or $\inf f>0$. If $u$ is a non-negative viscosity solution of \eqref{1.1}, then there exists a universal constant $c>0$, depending only on dimension and $\gamma$, such that
	\[ \sup_{B_{r}(x_0)} u \ge cr^{\frac{4}{3-\gamma}},  \]
where $x_0\in\overline{\{u>0\}}\cap\Omega$ and $0<r<\operatorname{dist}(x_0, \partial\Omega)$.
\end{theorem}
\begin{proof}
	Since $u$ is continuous, it is enough to prove the theorem for points $x_0\in\{u>0\}\cap\Omega$. Set
	$$
	v(x):=c|x-x_0|^{\frac{4}{3-\gamma}},
	$$
	with a constant $c\in(0,\Upsilon)$, where $\Upsilon>0$ is defined by \eqref{3.10}. Using \eqref{5.1}, direct computation reveals that the choice of $c$ makes $v$ a viscosity super-solution of \eqref{1.1} in $B_r(x_0)$, where $r>0$ is such that $B_r(x_0)\subset\Omega$. If $v\ge u$ on $\partial B_r(x_0)$, then the comparison principle, Lemma \ref{l2.1}, would imply $v\ge u$ in $B_r(x_0)$, contradicting to the fact that $0=v(x_0)<u(x_0)$. Hence, there is a point $y\in\partial B_r(x_0)$ such that $v(y)<u(y)$. We then estimate
	\[ \sup_{B_r(x_0)} u \geq u(y) \ge v(y) =cr^{\frac{4}{3-\gamma}}.\]	
\end{proof}
As a consequence, we obtain that the free boundary is a porous set, therefore it has Hausdorff dimension strictly less than $n$, hence its Lebesgue measure is zero.

Note that from \eqref{1.2} one has $f(0)=0$, so to use the comparison principle, Lemma \ref{l2.1}, it is enough to assume that $f$ is non-decreasing, that is, \eqref{comparison} holds.
\begin{corollary}\label{c4.1}
	Let \eqref{1.2}, \eqref{comparison} and \eqref{5.1} hold. If $u$ is a bounded non-negative viscosity solution of \eqref{1.1}, then $\partial\{u>0\}$ is a porous set.
\end{corollary}
\begin{proof}
	Let $x\in\partial\{u>0\}$ and $y\in\overline{B}_r(x)$ be such that
	$$
	u(y)=\sup_{B_{r}(x)}u.
	$$
	By Theorem \ref{t4.1}, $u(y)\ge cr^{\frac{4}{3-\gamma}}$. On the other hand, Theorem \ref{t3.1} provides
	$$
	u(y)\le C\left[d(y)\right]^{\frac{4}{3-\gamma}},
	$$
	where $d(y):=\text{dist}\left({y,\partial\{u>0\}}\right)$. Therefore,
	$$
	\left(\frac{c}{C}\right)^{\frac{3-\gamma}{4}}r\le d(y).
	$$
	Hence, if $\sigma:=\frac{1}{2}\left(\frac{c}{C}\right)^{\frac{3-\gamma}{4}}$, one has $$
	B_{2\sigma r}(y)\subset\{u>0\}.
	$$
	We now choose $\xi\in(0,1)$ such that for the point $z:=\xi y+(1-\xi)x$ we have $|y-z|=\sigma r$. Then
	$$
	B_{\sigma r}(z)\subset B_{{2\sigma}r}(y)\cap B_r(x).
	$$
	Moreover, we have
	$$
	B_{2\sigma r}(y)\cap B_r(x)\subset \{u>0\},
	$$	
	which together with the previous inclusion implies
	$$
	B_{\sigma r}(z)\subset B_{2\sigma r}(y)\cap B_r(x)\subset B_r(x)\setminus\partial\{u>0\},
	$$
	that is, the set $\partial\{u>0\}$ is porous with porosity $\sigma$.
\end{proof}
\begin{corollary}\label{c4.2}
	If \eqref{1.2}, \eqref{comparison}, \eqref{5.1} hold, and $u$ is a viscosity solution of \eqref{1.1}, then Lebesgue measure of the set $\partial\{u>0\}$ is zero.
\end{corollary}
\section{The borderline case}\label{s5}
Although, in general, one cannot expect more than $C^{1,\alpha}$ regularity for viscosity solutions of \eqref{1.1}, Theorem \ref{t3.1} provides higher and higher regularity across the free boundary, as $\gamma\in[0,3)$ gets closer to 3. In this section we analyze the limit case of $\gamma=3$. The scaling property of the operator plays an essential role here, as $\gamma=3$ is also the degree of homogeneity of the infinity Laplacian, meaning that $\Delta_\infty(Cu)=C^3\Delta_\infty u$, for any constant $C$. Observe that Theorem \ref{t3.1} cannot be applied directly, since the estimates deteriorate as $\gamma\to3$. Thus, in this section \eqref{1.2} is substituted by
\begin{equation}\label{6.1}
0\le f(\delta t)\le M\delta^3f(t),
\end{equation}
with $M>0$, $t>0$ bounded and $\delta>0$ small. Our first observation states as follows.
\begin{lemma}\label{l5.1}
	If $u$ is a non-negative viscosity solution of \eqref{1.1}, where $f$ satisfies \eqref{6.1}, then its every zero is of infinite order.
\end{lemma}
\begin{proof}
	This is a consequence of Theorem \ref{t3.1}. To see that it is enough to rewrite \eqref{1.2}, for $\gamma=3$, as
	$$
	f(\delta t)\leq M_\delta\delta^{3-\beta}f(t),
	$$
	where $M_\delta:=M\delta^\beta$ and $\beta>0$. An application of Theorem \ref{t3.1} with $M=M_\delta$ leads to the conclusion that if $u(z)=0$ for $z\in\Omega$, then $D^nu(z)=0$, $\forall n\in\mathbb{N}$.
\end{proof}
Furthermore, we show that if a non-negative viscosity solution of \eqref{1.1} vanishes at a point, then it must vanish everywhere. For $f\equiv0$ this follows from the Harnack inequality. The particular case, when $f$ is homogeneous of degree three, that is, $f(t)=M t^3$, was studied in \cite{ALT16}, where by means of a suitable barrier function, was concluded that if non-negative viscosity solution vanishes in an inner point, then it has to vanish everywhere. Unlike \cite{ALT16}, our function $f$ is not given explicitly, which makes the construction of a suitable barrier function more complicated. Observe that once \eqref{6.1} holds, then $f(0)=0$, hence $\inf f=0$, so to use the comparison principle, one needs to assume that $f$ is non-decreasing.
\begin{theorem}\label{t5.1} Let $u$ be a non-negative viscosity solution of \eqref{1.1}, where $f$ satisfies \eqref{comparison} and \eqref{6.1}. If $\{u=0\}\cap\Omega\neq\emptyset$, then $u\equiv0$.
\end{theorem}
\begin{proof}
	We argue by contradiction, assuming that there is $x\in\Omega$ such that $u(x)=0$, but $u(y)>0$ for a point $y\in\Omega$. Without loss of generality we may assume that
	$$
	r:=\dist\left(y,\{u=0\}\right)<\frac{1}{10}\dist\left(y,\partial\Omega\right).
	$$
	We aim to construct a sub-solution of \eqref{1.1} which stays below $u$ on $\partial B_r(y)$.
	
	Let $w$ be an infinity sub-harmonic function in $B_r(y)$ such that $|\nabla w|\ge\eta$ for $\eta\ge0$ constant to be chosen later. Such function can be built up as a limit, as $p\to\infty$, of $p$-super-harmonic functions with modulus of gradient separated from zero by $\eta$. We refer the reader for details to \cite{J93}. 

    Now if $g$ is a smooth function and $v=g(w)$, direct computation reveals that
    $$
    \Delta_{\infty} v=\left[g'(w)\right]^3\Delta_{\infty} w+\left[g^{\prime}(w)\right]^{2} g^{\prime \prime}(w)|\nabla w|^{4}.
    $$
    Thus, for $g(t)=e^t+t$,
    \begin{equation}\label{5.2}
    \Delta_{\infty} v\ge\left[g^{\prime}(w)\right]^{2} g^{\prime \prime}(w)|\nabla w|^{4},
    \end{equation}
    since $g^\prime\ge1$ and $\Delta_\infty w\ge0$. Also $g^{\prime\prime}\ge e^{-\|w\|_\infty}>0$, and \eqref{5.2} yields (recall that $|\nabla w|\ge\eta$)
    \begin{equation}\label{5.3}
    	\Delta_\infty v\ge\mu\eta,
    \end{equation}
    where $\mu:=e^{-\|w\|_\infty}>0$. Choosing
    $$
    \eta>\frac{M}{\mu}\max_{[0,\|v\|_\infty]}f,
    $$
    from \eqref{5.3} we obtain
    $$
    \Delta_\infty v-Mf(v)\ge\Delta_\infty v-\mu\eta\ge0,
    $$
    i.e., $v$ is a sub-solution of \eqref{1.1}. The latter together with \eqref{1.2} gives, for any small constant $\delta>0$,
	$$
	\Delta_\infty\left(\delta v\right)-f\left(\delta v\right)\ge\delta^3\left(\Delta_\infty v-Mf(v)\right)\ge0,
	$$	
	that is, the function $\delta v$ is also a sub-solution of \eqref{1.1}. We choose $\delta>0$ small enough to guarantee
	$$
		\delta v(x)\le u(x),\,\,\,x\in\partial B_r(y),
	$$
	and by the comparison principle, Lemma \ref{l2.1},
	\begin{equation}\label{5.4}
		\delta v(x)\le u(x),\,\,\,x\in B_r(y).
	\end{equation}
    Observe, that writing \eqref{1.2} as
    $$
    f(\delta t)\le M\delta\delta^2f(t),
    $$
    and applying Theorem \ref{t3.1} with $\widetilde{M}=M\delta$, we arrive at
    \begin{equation}\label{5.5}
    	\sup_{B_d(z)}u\le Cd^4,
    \end{equation}
    where $z\in\partial B_r(y)\cap\partial\{u>0\}$, and $d>0$ is small. In fact, we choose $0<d<\left(\frac{\delta\eta}{4C}\right)^\frac{1}{3}$. Using the fact that $|\nabla v|=g^\prime|\nabla w|\ge\eta$, recalling \eqref{5.4} and \eqref{5.5} and the choice of $d$, we estimate
    $$
    \delta\eta d\le\sup_{B_d(z)}\delta|v(x)-v(z)|\le\sup_{B_d(z)}\delta v\le\sup_{B_d(z)}u\le Cd^4\le\frac{1}{4}\delta\eta d,
    $$
    which is a contradiction.
\end{proof}

\section{Examples and beyond}\label{s7}
We close the paper with some examples of the source term, for which our results are valid. We start with the following remark.

\begin{remark}\label{r7.1}
	The results in this paper remain true, without changes in the proofs, when in \eqref{1.1} the right hand side has continuous (bounded) coefficients, i.e., 
	$$
	\Delta_\infty u=f(x,u),
	$$ 
	as long as $f(x,u)$ satisfies \eqref{1.2} (and \eqref{comparison}, \eqref{6.1} when needed) as a function of $u$.
\end{remark}

Note that $f\equiv0$ satisfies \eqref{1.2}, \eqref{comparison}, \eqref{5.1}, therefore, our results resemble those for the non-negative infinity harmonic functions. Also, the results are true when in \eqref{1.1} one has $f(t)=t^\gamma$, $t\ge0$, $\gamma\in[0,3)$ (studied in \cite{ALT16}). In view of Remark \ref{r7.1}, we can allow some coefficients too, as long as they remain bounded, that is, functions of the type $f(x,t)=g(x)t^\gamma$, where $g$ is a non-negative continuous, bounded function. When in the last example $\gamma=0$, i.e., the source term depends only on $x$, $f(x,t)=g(x)$, across the free boundary one has $C^{\frac{4}{3}}$ regularity, once $g(x)\ge0$ is continuous and bounded. Furthermore, the non-degeneracy result is true, and hence the free boundary is a porous set and has zero Lebesgue measure.

An example of the source term, not constructed via power functions is $f(t):=e^t-1$, which satisfies \eqref{1.2} with $M=1$ and $\gamma=0$, since $e^{\delta t}<e^t$, for $\delta>0$ small and $t>0$. Hence, applying Theorem \ref{t3.1} to
$$
\Delta_\infty u=e^u-1,
$$
we conclude that non-negative viscosity solutions of the above equation are $C^\frac{4}{3}$ near the free boundary $\partial\{u>0\}$. Moreover, if $u$ is an entire solution of the last equation, which vanishes at a point and
$$
u(x)=o\left(|x|^\frac{4}{3}\right),\,\,\,\textrm{ as }\,\,\,|x|\rightarrow\infty,
$$
then Theorem \ref{t3.2} implies that it has to be identically zero.

Same conclusion can be made when in \eqref{1.1} the source term is $f(t):=\log (t^2+1)$. Of course, Theorem \ref{t3.1} can be applied also for any linear combination (with continuous, bounded coefficients) of the above source terms. In fact, Theorem \ref{t3.1} is true for any non-negative continuous function, which is non-decreasing around zero and vanishes at the origin. We point out that \eqref{1.2} (as well as \eqref{5.1} and \eqref{6.1}) is required to hold only around zero, so Theorem \ref{t3.1} is true also for source terms that are any of the above examples around zero and can be anything outside, while remaining non-negative and continuous (and in case of Theorem \ref{t4.1} and its consequences, also non-decreasing).

We finish with an application of Theorem \ref{t5.1}. Let $u$ be a non-negative viscosity solution of
\begin{equation}\label{7.1}
	\Delta_\infty u=\log\left(1+u^3\right).
\end{equation}
Since $f(t)=\log(1+t^3)$ satisfies \eqref{1.2} with $M=1$ and $\gamma=0$, Theorem \ref{t3.2} implies $C^{\frac{4}{3}}$ regularity of $u$ near the touching ground. On the other hand, the function $f(t)$ can be written as 
$$
\log\left(1+t^3\right)=t^3\frac{\log\left(1+t^3\right)}{t^3}:=t^3g(t).
$$ 
Set $g(0)=1$. Then $g$ is continuous, bounded ($0\le g\le1$) function, and the non-decreasing function $f(t)=g(t)t^3$ satisfies \eqref{6.1} with $M=1$. Applying Theorem \ref{t5.1}, we conclude that if $u$ is a non-negative viscosity solution of \eqref{7.1}, which is zero at a point, then it must be identically zero.

\bigskip

\noindent{\bf Acknowledgments.} NMLD was partially supported by Instituto Federal de Educa\c{c}\~ao, Ci\^encia e Tecnologia do Rio Grande do Sul. NMLD thanks the Analysis group at Centre for Mathematics of the University of Coimbra (CMUC) for fostering a pleasant and productive scientific atmosphere during his postdoctoral program. RT was partially supported by FCT -- Funda\c c\~ao para a Ci\^encia e a Tecnologia, I.P., through projects PTDC/MAT-PUR/28686/2017 and UTAP-EXPL/MAT/0017/2017, and by CMUC -- UID/MAT/00324/2013, funded by the Portuguese government through FCT and co-funded by the European Regional Development Fund through Partnership Agreement PT2020.

\end{document}